\documentclass{article}
\usepackage[utf8]{inputenc}
\usepackage{authblk}
\usepackage{mathtools,bm}
\usepackage{color}
\usepackage{xcolor}
\usepackage[section]{placeins}

\usepackage{latexsym,upgreek}
\usepackage{amsfonts,psfrag,amsmath,cancel,epstopdf,multirow,algorithm,algpseudocode,url,caption,threeparttable,ragged2e,wrapfig,overpic,subcaption,graphicx,amsthm,amssymb}

\usepackage[colorlinks=true, pdfstartview=FitV, linkcolor=blue,citecolor=blue, urlcolor=blue]{hyperref}
\newcommand*{\email}[1]{\href{mailto:#1}{\nolinkurl{#1}}}

\title{A Two-Level Galerkin Reduced Order Model for the Steady Navier-Stokes Equations}
\author[1]{Dylan Park}
\author[2]{Changhong Mou}
\author[1]{Honghu Liu}
\author[3]{Adrian Sandu}
\author[1]{Traian Iliescu}
\affil[1]{Department of Mathematics, Virginia Tech, Blacksburg, VA 24061
{\em (\email{dylantp@vt.edu}, \email{hhliu@vt.edu}, \email{iliescu@vt.edu})}}
\affil[2]{University of Wisconsin-Madison, Madison, WI 53706 {\em (\email{cmou3@wisc.edu})}}
\affil[3]{Department of Computer Science, Virginia Tech, Blacksburg, VA 24061 {\em (\email{sandu@cs.vt.edu})}}

\date{\today}

\newtheorem{remark}{Remark}[section]
\newtheorem{lemma}{Lemma}[section]
\newtheorem{theorem}{Theorem}[section]

\newtheorem{proposition}{Proposition}[section]

\def\PP{{{\rm l}\kern - .15em {\rm P} }}
\def\PN2{{\PP_{N}-\PP_{N-2}}}

\newcommand{\cE}{\mathcal{E}}

\definecolor{vargreen}{rgb}{0.0, 0.5, 0.0}

\newcommand{\deleted}[1]{{}}

\begin{document}

\graphicspath{{../figures/}}

\maketitle

\begin{abstract}
We propose, analyze, and investigate numerically a novel two-level Galerkin reduced order model (2L-ROM) for the efficient and accurate numerical simulation of the steady Navier-Stokes equations.
In the first step of the 2L-ROM, a relatively low-dimensional nonlinear system is solved.
In the second step, the Navier-Stokes equations are linearized around the solution found in the first step, and a higher-dimensional system for the linearized problem is solved.
We prove an error bound for the new 2L-ROM and compare it to the standard one level ROM (1L-ROM) in the numerical simulation of the steady Burgers equation.
The 2L-ROM significantly decreases (by a factor of $2$ and even $3$) the 1L-ROM computational cost, without compromising its numerical accuracy.

\end{abstract}

\section{Introduction}
    \label{sec:introduction}

Two-level methods have been used to reduce the computational cost of classical numerical methods, e.g., the finite element method (FEM), for nonlinear problems, such as the Navier-Stokes equations~\cite{axe961,dawson1998two,layton1993two,layton1999two,Xu94,Xu96}.
In the FEM setting, the two-level methods can be summarized as follows: 
\begin{itemize}
    \item[(I)] In the first step, solve the nonlinear problem on a {\it coarse} mesh.
    \item[(II)] In the second step, linearize the nonlinear problem around the solution obtained in step (I), and then solve the resulting {\it linear} problem on the {\it fine} mesh.
\end{itemize}

The {\it computational cost} of the two-level method is significantly lower than the cost of the standard, one-level method (i.e., solving the nonlinear problem on the fine mesh).
Indeed, the one-level method is generally expensive, since it solves the nonlinear problem on the fine mesh, which can require many (e.g., $\mathcal{O}(10)-\mathcal{O}(10^2)$) iterations in a nonlinear solver (e.g., a Newton iteration). 
In contrast, the two-level method solves the nonlinear problem only on a coarse mesh, which is much more efficient than solving it on the fine mesh.
Of course, the two-level method also solves a linear system on the fine mesh, but its cost is much lower than the cost of solving the nonlinear problem on the fine mesh.

Not only is the two-level method significantly more efficient than the standard one-level method, but it is also as accurate as the latter.
Indeed, by performing a rigorous numerical analysis and  carefully choosing the scaling between the coarse and fine mesh sizes, one can ensure that {\it the convergence rates of the two-level and one-level methods are the same}~\cite{layton1993two}. 

Two-level methods have been successfully used for efficient numerical simulations of a wide variety of challenging nonlinear problems, e.g., the Navier-Stokes equations~\cite{fairag1998two,fairag2003numerical,layton1993two}, the Smagorinsky model used in large eddy simulation of turbulent flows~\cite{borggaard2008two}, the quasi-geostrophic equations (QGE) modeling the large scale ocean circulation~\cite{foster2013two,foster2016conforming}, and viscoelastic fluid flows~\cite{liakos2003two}. 

Despite their success in a FEM setting, to our knowledge, {\it the two-level method has not been used in a reduced order modeling setting}.
In this paper, we take a step toward filling this gap and propose a two-level method for Galerkin reduced order models of the steady Navier-Stokes equations.

{Galerkin reduced order models (G-ROMs)} are computational models that significantly decrease the dimension (and, thus, computational cost) of full order models (FOMs), i.e., models obtained from classical numerical discretizations (e.g., FEM).
The G-ROM dimension is generally $\mathcal{O}(10^3)-\mathcal{O}(10^4)$ lower than the FOM dimension.
(In special cases, the decrease can be even higher.)
When the original FOM dimension is $\mathcal{O}(10^4)-\mathcal{O}(10^5)$, the corresponding ROM is a low-dimensional model that can be efficiently run on a laptop using standard numerical discretizations.
We emphasize, however, that in challenging applications (e.g., turbulent flows, combustion, nuclear engineering, and geophysics), FOMs can require {\it billions and even hundreds of billions of degrees of freedom}~\cite{fischer2022nekrs}.
In those cases, the corresponding ROMs are {\it not} low-dimensional: they can require hundreds or thousands (or even more) basis functions (see, e.g, Table II in~\cite{ahmed2021closures} for examples in the numerical simulation of the atmospheric boundary layer).
Thus, in these settings, the {\it efficient} and accurate numerical discretization of these relatively high-dimensional G-ROMs becomes {\it critical}.

In this paper, we propose, analyze, and investigate numerically a new two-level G-ROM for the efficient and accurate numerical simulation of the steady Navier-Stokes equations (and related systems, e.g., Burgers equation).
To our knowledge, the first two-level G-ROM was proposed in~\cite{wang2011two} for the efficient numerical simulation of a ROM closure model~\cite{ahmed2021closures} (i.e., the Smagorinsky model).
This strategy was later used in~\cite{gaonkar2015application}.
Recently, a nonintrusive two-grid G-ROM was proposed in~\cite{grosjean2022variations,grosjean2022doubly,grosjean2022error}.
A two-grid approach for G-ROM adaptivity was used in~\cite{dai2020two}.
A two-level G-ROM approach to construct spatial and temporal basis functions was used in~\cite{audouze2013nonintrusive,chen2018greedy}.

We emphasize that our two-level G-ROM is fundamentally different from the approaches used in~\cite{audouze2013nonintrusive,chen2018greedy,dai2020two,gaonkar2015application,grosjean2022variations,grosjean2022doubly,grosjean2022error,wang2011two}.
Indeed, these approaches use two spatial (and/or temporal) grids.
In contrast, the two-level G-ROM proposed in this paper uses just one spatial grid and one temporal grid, and two ROM spaces.
Thus, the approaches in~\cite{audouze2013nonintrusive,chen2018greedy,dai2020two,gaonkar2015application,grosjean2022variations,grosjean2022doubly,grosjean2022error,wang2011two} are ``two-grid" ROMs, whereas the strategy proposed in this paper is a ``two-space" ROM.

The rest of the paper is organized as follows:
In Section~\ref{sec:preliminaries}, we present the mathematical formulation, the ROM basis construction, and the standard one-level G-ROM.
In Section~\ref{sec:2l-rom}, we outline the construction of the novel two-level G-ROM.
In Section~\ref{sec:theoretical-results}, we prove an error bound for the two-level G-ROM.
In Section~\ref{sec:numerical-results}, we compare the new two-level G-ROM with the standard one-level G-ROM in the numerical simulation of the steady Burgers equation.
Finally, in Section~\ref{sec:conclusions}, we draw conclusions and outline future research directions.

\section{
Preliminaries
}
    \label{sec:preliminaries}

We will present the theoretical investigation of the two-level G-ROM approach using the steady Navier-Stokes equations (NSE), which describe the flow of an incompressible fluid. In this section, we recall the corresponding weak formulation and the classical one-level G-ROM of the NSE based on proper orthogonal decomposition (POD)~\cite{HLB96,volkwein2013proper}.

\subsection{Mathematical Formulation}
We consider the steady NSE posed on a bounded domain $\Omega \subset \mathbb{R}^{d}$ ($d = 2, 3$) with Lipschitz-continuous boundary: 
\begin{align}
    \label{eqn:NSE}
    \begin{split}
        -Re^{-1} \Delta u + (u \cdot \nabla)u + \nabla p & = f, \quad \textnormal{in } \Omega, \\
        \nabla \cdot u & = 0, \quad \textnormal{in } \Omega, \\
        u & = 0, \quad \textnormal{on } \partial \Omega.
    \end{split}
\end{align}

The weak formulation of \eqref{eqn:NSE} makes use of the following function spaces (see,  e.g.,~\cite[Chapter 6]{layton2008introduction})
\begin{equation} \label{Eq_function_spaces}
\begin{aligned}
 & X \coloneqq (H_0^1(\Omega))^{d}, 
 \\
 & Q := L_0^2(\Omega) =  \left\{q \colon \Omega \rightarrow \mathbb{R} \; \vert \; q \in L^2(\Omega) \text{ and } \int_{\Omega}
 q(x)\,dx = 0 \right\},
\end{aligned}
\end{equation}
where
\begin{equation}
H_0^1(\Omega) = \left\{v \colon \Omega \rightarrow \mathbb{R} \; \vert \; v \in L^2(\Omega), \nabla v  \in (L^2(\Omega))^{d}, \text{ and } v = 0 \text{ on } \partial \Omega \right\}.
\end{equation} 
These spaces, $X$ for the velocity and $Q$ for the pressure, are endowed with the usual Sobolev norms and inner products. 
In the following, we use $\|\cdot\|_0$ to denote the $L^2$-norm, modified in the usual manner for vector-valued functions, and we denote the semi-norm $\|\nabla v\|_0$ of any $v$ in $(H^1(\Omega))^{d}$ by $|v|_1$. 
Note that since the pressure $p$ is determined up to an additive constant, it is normalized to have mean value zero by requiring $\int_{\Omega}p(x)\,dx = 0$. We seek $(u,p) \in (X,Q)$ satisfying the following system for all $(v,q) \in (X,Q)$:
\begin{subequations}
    \label{eqn:NSE_Weak_Formulation}
    \begin{align}
        & a(u,v) + b_{SKEW}(u;u,v) + c(p,v) = (f,v), \label{eqn:NSE_Weak_Formulation_part1}\\
        & c(q,u) = 0, \label{eqn:NSE_Weak_Formulation_part2}
        \end{align}
\end{subequations}
where 
\begin{align}
  &  a(u,v) \coloneqq Re^{-1}\int_{\Omega} (\nabla u):(\nabla v) \,dx, \\
   & b(u;v,w) \coloneqq \int_{\Omega} (u \cdot \nabla)v \cdot w \,dx, \label{Def_b_term}\\
   & b_{SKEW}(u;v,w) \coloneqq \frac{1}{2}[b(u;v,w) - b(u;w,v)], \label{Def_b_skew_term} \\
    & c(p,v) \coloneqq 
    - \int_{\Omega}(\nabla \cdot v)p \,dx, \quad \textnormal{and} \quad (f,v) \coloneqq 
    \int_{\Omega}f\cdot v \,dx.
\end{align}

\subsection{ROM Basis} \label{Sec_ROM_basis}
To build the ROM basis, we assume that we have access to the numerical solutions of \eqref{eqn:NSE_Weak_Formulation} for a set of Reynolds numbers, $\{Re_1, \ldots, Re_M\}$, within a given interval. We assume that such solutions are obtained based on a 
FEM discretization of \eqref{eqn:NSE_Weak_Formulation}, which satisfies the {\it discretely divergence-free condition} for the velocity. 

We denote such FEM solutions as 
\begin{equation} \label{Eq_FEM_soln}
\{u^1_h, \ldots, u^M_h\},
\end{equation}
where each $u^j_h$, also called a snapshot hereafter, denotes the computed solution associated with the $j$-th Reynolds number, $Re_j$, for a given FE mesh size $h$. 
The number of snapshots, $M$, is an arbitrarily fixed sufficiently large positive integer. The ROM basis used in this article 
consists of the POD modes \cite{HLB96,volkwein2013proper} constructed from the above set of snapshots. 

Denoting the obtained orthonormal ROM basis by $\{\varphi_1, \ldots, \varphi_{\ell}\}$, where $\ell$ is the number of linearly independent snapshots in \eqref{Eq_FEM_soln}, we obtain the ROM space $X^{\ell}$ defined as follows:
\begin{equation} \label{Eq_ROM_space}
X^{\ell} := \mathrm{span}\{\varphi_1, \ldots, \varphi_{\ell} \}.
\end{equation} 
Note that $X^{\ell}$ corresponds to the same space spanned by the given snapshots in \eqref{Eq_FEM_soln}.

\subsection{One-Level Method}
With the ROM basis at hand, the classical Galerkin ROM for a given dimension $R \le \ell$ can be readily constructed and is given by \eqref{eqn:Alg_Nonlin_1L} in Algorithm~\ref{alg:one_level_ROM}.  Thanks to the discretely divergence-free condition assumed for the FEM, the pressure term $c(p,v)$ in \eqref{eqn:NSE_Weak_Formulation} vanishes in the Galerkin ROM since $c(p,v) = 0$ for all $v$ in $X^{\ell}$ (see, e.g.,~\cite{hesthaven2015certified,quarteroni2015reduced} for alternative approaches). To distinguish it from 
the two-level ROMs proposed in 
Section~\ref{sec:2l-rom}, we will call the classical Galerkin ROM given by \eqref{eqn:Alg_Nonlin_1L} the $R$-dimensional one-level ROM (1L-ROM).
To distinguish it from the two-level ROM solution, we denote the one-level ROM solution with $u_{1L}^{R}$.

\begin{algorithm}[tbh]
\caption{One-Level ROM (1L-ROM) Algorithm}
\label{alg:one_level_ROM}
Solve the following nonlinear system with $R$ POD basis
functions. We seek $u_{1L}^R \in X^R:= \mathrm{span}\{\varphi_1, \ldots, \varphi_R\}$ satisfying for all $v^R \in X^R$
\begin{align}
    \label{eqn:Alg_Nonlin_1L}
    \begin{split}
        a(u_{1L}^R,v^R) + b_{SKEW}(u_{1L}^R;u_{1L}^R,v^R) = (f,v^R). \\
    \end{split}
\end{align}
\end{algorithm}

\section{Two-Level Method}
    \label{sec:2l-rom}

In challenging applications (e.g., turbulent flows, combustion, nuclear engineering, and geophysics), the standard 1L-ROM given by \eqref{eqn:Alg_Nonlin_1L} (i.e., the Galerkin ROM) can require hundreds or thousands (or even more) basis functions (see, e.g, Table II in~\cite{ahmed2021closures}).
In these settings, 1L-ROM can require the solution of relatively high-dimensional nonlinear systems whose operators are generally not sparse.
Thus, straightforward numerical discretizations of 1L-ROM can become impractical for these types of applications.

In this paper, we propose a novel {\it two-level ROM (2L-ROM)} that {\it significantly decreases the computational cost} of the standard 1L-ROM {\it and has the same convergence rate} as the 1L-ROM.
Next, we outline the 2L-ROM construction.

First, we emphasize that, in stark contrast to the standard 1L-ROM, which uses a single ROM subspace $X^R$, the new 2L-ROM utilizes 
{\it two nested ROM subspaces} $X^r \subset X^R$ with $r < R$.
The new 2L-ROM can be summarized as follows:
\begin{itemize}
    \item[(I)] In the first step, solve the nonlinear problem \eqref{eqn:Alg_Nonlin_1L} in the {\it low-dimensional} ROM subspace $X^r$ to obtain $u^r$.
    \item[(II)] In the second step, linearize the $R$-dimensional version of the Galerkin ROM \eqref{eqn:Alg_Nonlin_1L} with respect to $u^r$, and then solve the resulting {\it linear} problem in the {\it high-dimensional} ROM space $X^R$ to obtain $u^R$.
\end{itemize}
The 2L-ROM is formalized in Algorithm~\ref{alg:two_level_ROM}.

\begin{algorithm}[tbh]
\caption{Two-Level ROM (2L-ROM) Algorithm}
\label{alg:two_level_ROM}
Step 1. Solve the following nonlinear, low-dimensional problem, with $r$ POD basis
functions: We seek $u^r \in X^r$ satisfying for all $v^r \in X^r$ 
\begin{align}
    \label{eqn:Alg_Nonlin_2L}
    \begin{split}
        a(u^r,v^r) + b_{SKEW}(u^r;u^r,v^r) = (f,v^r). \\
    \end{split}
\end{align}
Step 2. Solve the following linear, higher-dimensional problem, with $R$ POD basis functions: Seek $u^R \in X^R$ satisfying for all $v^R \in X^R$
\begin{align}
    \label{eqn:Alg_lin_2L}
    \begin{split}
        a(u^R,v^R) + b_{SKEW}(u^R;u^r,v^R) +b_{SKEW}(u^r;u^R,v^R) \\
        = (f,v^R) + b_{SKEW}(u^r;u^r,v^R).
    \end{split}
\end{align}
\end{algorithm}

The {\it computational cost} of the 2L-ROM in Algorithm~\ref{alg:two_level_ROM} is significantly lower than the cost of the standard 1L-ROM in Algorithm~\ref{alg:one_level_ROM}.
Indeed, the 1L-ROM is generally expensive, since it solves a {\it high-dimensional} nonlinear problem, which can require many (e.g., $\mathcal{O}(10)-\mathcal{O}(10^2)$) iterations in a nonlinear solver (e.g., a Newton iteration). 
In contrast, the new 2L-ROM solves a {\it low-dimensional} nonlinear problem, which is much more efficient than solving the high-dimensional nonlinear problem.
Of course, the 2L-ROM also solves a high-dimensional linear system (in Step 2 of Algorithm~\ref{alg:two_level_ROM}), but its cost is much lower than the cost of solving the high-dimensional nonlinear problem in the 1L-ROM.
Thus, overall, the expectation is that the computational cost of the new 2L-ROM is significantly lower than the computational cost of the standard 1L-ROM.

\begin{remark}[Two-Space vs. Two-Grid]
We emphasize that all the current two-level ROMs are in fact two-grid methods.
To our knowledge, the 2L-ROM method in Algorithm~\ref{alg:two_level_ROM} is the first two-space ROM. 
\end{remark}

\section{Theoretical Results}
    \label{sec:theoretical-results}
 
We present in this section a basic error 
bound for the 2L-ROM solutions. The main result is summarized in Theorem~\ref{thm:main_them} below. As a preparation, we first establish a sufficient condition to ensure the existence and uniqueness of solutions for the 2L-ROM system.

\subsection{Well-Posedness Results}  \label{Sec_wellposdeness}

Let us recall the following classical result on the existence and uniqueness of solution to the full system \eqref{eqn:NSE_Weak_Formulation}. Related results for the 2L-ROM \eqref{eqn:Alg_Nonlin_2L}--\eqref{eqn:Alg_lin_2L} will be presented thereafter.

For this purpose, we define the following quantities for the trilinear term
\begin{equation}\label{Eq_def_constants}
\begin{aligned} 
    N & \coloneqq \sup_{u,v,w \in X\setminus \{\bm{0}\}} \frac{|b_{SKEW}(u;v,w)|}{|u|_1|v|_1|w|_1}, \\
    N_r & \coloneqq \sup_{u^r,v^r,w^r \in X^r\setminus \{\bm{0}\}} \frac{|b_{SKEW}(u^r;v^r,w^r)|}{|u^r|_1|v^r|_1|w^r|_1}, \\
    N_{R,r} & \coloneqq \sup_{\substack{u^R,w^R \in X^R\setminus \{\bm{0}\} \\ v^r \in X^r\setminus \{\bm{0}\}}} \frac{|b_{SKEW}(u^R;v^r,w^R)|}{|u^R|_1|v^r|_1|w^R|_1}, \\
    N_{r,R} & \coloneqq \sup_{\substack{u^r,v^r \in X^r\setminus \{\bm{0}\} \\ w^R \in X^R\setminus \{\bm{0}\}}} \frac{|b_{SKEW}(u^r;v^r,w^R)|}{|u^r|_1|v^r|_1|w^R|_1}, \textnormal{ and} \\
    |f|_* & \coloneqq \sup_{v\in X\setminus \{\bm{0}\}}\frac{|(f,v)|}{|v|_1},
\end{aligned}
\end{equation}
where the space $X$ is defined by \eqref{Eq_function_spaces}, $X^r$ and $X^R$ are subspaces of $X$ spanned by the leading POD modes with $r < R \le \ell$ and with $\ell$ given by \eqref{Eq_ROM_space}, and $\bm{0}$ denotes the equivalence class with norm zero in each of the underlying function spaces. 

\begin{proposition}  \label{prop:2.1}
   Let $d = 2$ or $3$ and $\Omega$ be a bounded domain in $\mathbb{R}^d$ with Lipschitz-continuous boundary $\partial \Omega$. 
   Then, for any $f$ in the dual space of $X$ and any Reynolds number $Re > 0$, the weak formulation~\eqref{eqn:NSE_Weak_Formulation} admits at least one solution $(u,p) \in (X, Q)$. Moreover, if the data are chosen such that for $N$ and $|f|_*$ given in \eqref{Eq_def_constants}, it satisfies 
    \begin{align}
        \label{eqn:prop}
        Re^2N|f|_* < 1,
    \end{align}
    then the solution to \eqref{eqn:NSE_Weak_Formulation} is also unique.
\end{proposition}

See, e.g.~\cite[Proposition 9 on page 107 and Theorem 17 on page 113]{layton2008introduction} for a proof of the above results,  which is based on the Leray-Schauder fixed point theorem. See also \cite[Chapter IV, Theorems 2.1 and 2.2]{GR86} or \cite[Chapter IX.3]{galdi2011introduction}, which uses instead a Galerkin method approach in combination with the fixed point theorem of Brouwer for each of the finite-dimensional approximation problem and then passes to the limit.  

Likewise, we have the following analogous results for the first equation, \eqref{eqn:Alg_Nonlin_2L}, in the 2L-ROM algorithm, which 
simply consists of the $r$-dimensional 1L-ROM, as pointed out before.  
\begin{lemma}
\label{lemma:2.4}
    Solutions to \eqref{eqn:Alg_Nonlin_2L} exist and satisfy 
    \begin{equation} \label{Eq_1LROM_apriori_est}
    |u^r|_1 \leq Re|f|_*.
    \end{equation} 
    Moreover, if $Re^2N_r |f|_* < 1$ for $N_r$ defined in \eqref{Eq_def_constants}, then the solution is also unique. 
\end{lemma}

Note that the existence result and the solution bound \eqref{Eq_1LROM_apriori_est} presented in Lemma~\ref{lemma:2.4} are simply those obtained at the Galerkin approximation stage of the proof for \cite[Chapter IV, Theorems 2.1]{GR86}. The uniqueness result under the assumption $Re^2N_r |f|_* < 1$ follows the same lines of argument as presented in \cite[Proposition 9 on page 107]{layton2008introduction} and \cite[Chapter IV, Theorems 2.2]{GR86}.

With the solution to \eqref{eqn:Alg_Nonlin_2L} available, we have the following result regarding the solution to the second equation, \eqref{eqn:Alg_lin_2L}, in the 2L-ROM algorithm. 

\begin{lemma}
    \label{lemma:2.5}
Let $f$ be an arbitrary element in the dual space of $X$. Let $N_{R,r}$ and $|f|_*$ be defined in \eqref{Eq_def_constants}. Assume that
    \begin{align} \label{eqn:uniqueness_cond}
        Re^2N_{R,r}|f|_* < 1.
    \end{align}
    Then, for any given solution $u^r$ to \eqref{eqn:Alg_Nonlin_2L}, the solution $u^R$ to \eqref{eqn:Alg_lin_2L} exists, is unique and satisfies
    \begin{align} \label{ean:uR_bound}
        |u^R|_1 \leq Re(1-Re^2N_{R,r}|f|_*)^{-1}(1+Re^2N_{r,R}|f|_*)|f|_*.
    \end{align}
\end{lemma}
\begin{proof}
Note that \eqref{eqn:Alg_lin_2L} is a finite-dimensional linear algebraic system for $u^R$. To show the existence and uniqueness of a solution, it suffices to establish the \textit{a priori} bound \eqref{ean:uR_bound}. For this purpose, let us set $v^R = u^R$ in \eqref{eqn:Alg_lin_2L}  to get 
    \begin{equation}
    \begin{aligned}
a(u^R,u^R) + b_{SKEW}(u^R;u^r,u^R) & + b_{SKEW}(u^r;u^R,u^R) \\
& = (f,u^R) + b_{SKEW}(u^r;u^r,u^R).
    \end{aligned}
        \end{equation}
Since $b_{SKEW}(u^r;u^R,u^R) = 0$ due to skew symmetry and $a(u^R,u^R) = Re^{-1}|u^R|_1^2$, we get 
\begin{equation}
Re^{-1}|u^R|_1^2 - N_{R,r}|u^R|_1^2|u^r|_1 \leq (f,u^R) + N_{r,R}|u^r|_1^2|u^R|_1,  
        \end{equation}
where we have used the definition of $N_{R,r}$ and $N_{r,R}$ given in \eqref{Eq_def_constants} to estimate the terms $b_{SKEW}(u^R;u^r,u^R)$ and $b_{SKEW}(u^r;u^r,u^R)$.  Since $(f,u^R) \le |f|_*|u^R|_1$, we get 
\begin{equation}
Re^{-1}|u^R|_1 - N_{R,r}|u^R|_1|u^r|_1 \leq |f|_* + N_{r,R}|u^r|_1^2.  
\end{equation}

Now, by using the estimate \eqref{Eq_1LROM_apriori_est} for $|u^r|_1$ in the above inequality, we arrive at
\begin{equation}
        (Re^{-1} - ReN_{R,r}|f|_*)|u^R|_1 \leq |f|_* + Re^2N_{r,R}|f|_*^2.
\end{equation}
The above estimate shows that if $(1 - Re^2 N_{R,r}|f|_*) > 0$, then $|u^R|_1$ would satisfy the estimate \eqref{ean:uR_bound} provided that the solution $u^R$ exists. However, thanks to \eqref{ean:uR_bound}, we see that $|u^R|_1$ must be zero if $f$ is identically zero; namely, the finite-dimensional linear operator associated with the left-hand side (LHS) of \eqref{eqn:Alg_lin_2L}, which maps $u^R$ to $a(u^R,v^R) + b_{SKEW}(u^R;u^r,v^R) +b_{SKEW}(u^r;u^R,v^R)$ for each fixed $v^R$ in $X^R$, is invertible. Hence, \eqref{eqn:Alg_lin_2L} admits a unique solution for each $f$ and $u^r$. 
\end{proof}

\begin{remark} \label{Rmk:uniquess_cond}
Note that from \eqref{Eq_def_constants}, we see that $N_r \le N$ and $N_{R,r} \le N$ since the supremum in the definitions of $N_r$ and $N_{R,r}$ are taken in a smaller space than that for $N$. As a result, the condition $ Re^2N|f|_* < 1$ given by \eqref{eqn:prop} automatically implies that $Re^2N_r |f|_* < 1$ and $Re^2N_{R,r} |f|_* < 1$. Thus, condition \eqref{eqn:prop} ensures that both the steady NSE \eqref{eqn:NSE_Weak_Formulation} and the 2L-ROM system \eqref{eqn:Alg_Nonlin_2L}--\eqref{eqn:Alg_lin_2L} have a unique solution.

\end{remark}

\subsection{Error estimates}
With the above well-posedness results available, we 
derive 
a bound for the error between the solution $u$ to the steady NSE system \eqref{eqn:NSE_Weak_Formulation} and the solution $u^R$ to the 2L-ROM system obtained through Algorithm~\ref{alg:two_level_ROM}. We will first rewrite the weak formulation \eqref{eqn:NSE_Weak_Formulation} in a form that matches better the form of the linear system \eqref{eqn:Alg_lin_2L} in the 2L-ROM. The equation satisfied by the error $u-u^R$ is then derived, and the corresponding error 
bound is presented in Theorem~\ref{thm:main_them}.

\paragraph{Rewriting the weak formulation \eqref{eqn:NSE_Weak_Formulation}.}

Let us first recall the following lemma from \cite[Lemma 2.2]{layton1993two}, which concerns two basic identities about the trilinear term in the weak formulation \eqref{eqn:NSE_Weak_Formulation} of the NSE. Note that $u^r$ in the lemma below is an arbitrary element in $(H^1(\Omega))^d$. In the estimates 
proved later in this section, we will set $u^r$ to be the solution of the first equation in the 2L-ROM; cf.~\eqref{eqn:Alg_Nonlin_2L}.

 \begin{lemma}  \label{lemma:2.2}
    Let $b(u;v,w)$ and $b_{SKEW}(u;v,w)$ be given by \eqref{Def_b_term} and \eqref{Def_b_skew_term}, respectively. Then, for any $u,v,u^r \in (H^1(\Omega))^d$, it holds that
    \begin{equation}
        b(u;u,v) = b(u;u^r,v) + b(u^r;u,v) - b(u^r;u^r,v) + b(u-u^r;u-u^r,v), 
        \end{equation}
        and that
        \begin{equation} \label{Eq_b_skew_identity}
        \begin{aligned}
        b_{SKEW}(u;u,v) &= b_{SKEW}(u;u^r,v) + b_{SKEW}(u^r;u,v) \\
        				 & \qquad -b_{SKEW}(u^r;u^r,v) + b_{SKEW}(u-u^r;u-u^r,v).
	 \end{aligned}
\end{equation}
 \end{lemma}
The two identities in Lemma~\ref{lemma:2.2} are obtained by using the following identity in the definitions of $b$ and $b_{SKEW}$
 \begin{equation} \label{Eq_nonlin_identity}
        (u\cdot \nabla)u = -(u^r\cdot \nabla)u^r + (u \cdot \nabla)u^r + (u^r \cdot \nabla)u + [(u-u^r)\cdot \nabla](u-u^r).
\end{equation}
Note also that \eqref{Eq_nonlin_identity} itself results from a direct expansion of $[(u-u^r)\cdot \nabla](u-u^r)$.

Using \eqref{Eq_b_skew_identity} in Lemma \ref{lemma:2.2}, we can rewrite \eqref{eqn:NSE_Weak_Formulation_part1} in the weak formulation of the NSE as: For any $u^r \in (H^1(\Omega))^d$, 
seek $u \in X$ that satisfies for all $v$ in $X$ that
\begin{align}
    \label{eqn:NSE_Weak_Formulation_Equivalent}
    \begin{split}
        a(u,v) & + b_{SKEW}(u;u^r,v) + b_{SKEW}(u^r;u,v) \\
        & + b_{SKEW}(u-u^r;u-u^r,v) + c(p,v) = (f,v) + b_{SKEW}(u^r;u^r,v).
    \end{split}
\end{align}

\paragraph{Derivation of the error equation.} Instead of working directly on the error $u-u^R$, it will be beneficial to consider instead 
\begin{equation} \label{Eq_new_err_terms}
\phi^R := u^R - w^R,  \quad \eta^R := u - w^R, 
\end{equation}
where $w^R$ is an arbitrary element in $X^R$. The main effort is to derive a suitable bound for $|\phi|^R_1$. Once this is done, an error bound for $|u-u^R|_1$ can be obtained by using the triangle inequality.

To do so, we set $v$ in \eqref{eqn:NSE_Weak_Formulation_Equivalent} to be $v^R$ and by subtracting the resulting equation from \eqref{eqn:Alg_lin_2L}, we get
\begin{equation} \label{Err_eqn1}
\begin{aligned}
a(u^R,v^R) - a(u,v^R) & + b_{SKEW}(u^R; u^r,v^R)  - b_{SKEW}(u;u^r,v^R)  \\
& + b_{SKEW}(u^r;u^R,v^R)    - b_{SKEW}(u^r;u,v^R) \\
 &   - b_{SKEW}(u-u^r;u-u^r,v^R) - c(p,v^R) = 0.
\end{aligned}
\end{equation}
Now we rewrite the above equation in terms of $\phi^R$ and $\eta^R$ defined by \eqref{Eq_new_err_terms}. For instance, the 
terms $a(u^R,v^R) - a(u,v^R)$ can be rewritten as
\begin{equation}
\begin{aligned}
a(u^R,v^R) - a(u,v^R) & =  \big(a(u^R,v^R)  - a(w^R,v^R) \big)  - \big( a(u,v^R) - a(w^R,v^R) \big) \\
& = a(\phi^R,v^R) - a(\eta^R,v^R).
\end{aligned}
\end{equation}
Similarly, we have 
\begin{equation}
\begin{aligned}
& b_{SKEW}(u^R; u^r,v^R)  - b_{SKEW}(u;u^r,v^R) \\
& \qquad =  b_{SKEW}(\phi^R; u^r,v^R)  - b_{SKEW}(\eta^R;u^r,v^R),
\end{aligned}
\end{equation}
and 
\begin{equation}
\begin{aligned}
& b_{SKEW}(u^r;u^R,v^R)    - b_{SKEW}(u^r;u,v^R) \\
& \qquad = b_{SKEW}(u^r; \phi^R, v^R)  - b_{SKEW}(u^r;\eta^R,v^R).
\end{aligned}
\end{equation}
Using the above identities in \eqref{Err_eqn1}, we get 
\begin{equation} \label{Err_eqn2}
\begin{aligned}
 a(\phi^R,v^R) - a(\eta^R,v^R) & +  b_{SKEW}(\phi^R;u^r,v^R) - b_{SKEW}(\eta^R;u^r,v^R) \\
 & + b_{SKEW}(u^r;\phi^R,v^R)  - b_{SKEW}(u^r;\eta^R,v^R)  \\
  & - b_{SKEW}(u-u^r;u-u^r,v^R) - c(p,v^R) = 0.
\end{aligned}
\end{equation}

Note that since we assumed that the POD modes are constructed from FEM solution snapshots that are discretely divergence-free (cf.~Section~\ref{Sec_ROM_basis}), then for any $\chi^h$ in the discretized pressure space $Q^h$ 
involved in the construction of these FEM solutions (cf.~\eqref{Eq_FEM_soln}), 
we have 
\begin{equation}
c(\chi^h,v^R) = 0, \quad \text{for all } v^R \in X^R.
\end{equation}
In turn, we have 
\begin{equation} \label{Eq_pressure_rewritten}
c(p,v^R) = c(p,v^R) - c(\chi^h,v^R) = c(p-\chi^h,v^R), \; \text{for all } v^R \in X^R,  \chi^h \in Q^h.
\end{equation}

Using \eqref{Eq_pressure_rewritten} in \eqref{Err_eqn2}, we get the following error equation after rearranging the terms
\begin{equation} \label{eqn:Error_Eqn}
\begin{aligned}
        a(\phi^R,v^R) &+ b_{SKEW}(\phi^R;u^r,v^R) + b_{SKEW}(u^r,\phi^R,v^R) \\
         & = a(\eta^R,v^R)  + b_{SKEW}(\eta^R;u^r,v^R) + b_{SKEW}(u^r;\eta^R,v^R) \\
        & \hspace{5.5em} + b_{SKEW}(u-u^r;u-u^r,v^R) + c(p-\chi^h,v^R),
\end{aligned}
\end{equation}
which holds for all $v^R \in X^R$ and $\chi^h \in Q^h$, where $\phi^R$ and $\eta^R$ are defined in \eqref{Eq_new_err_terms}.

\paragraph{Estimation 
of the error $|u - u^R|_1$.}
We are now in a position to 
prove a bound for $|u - u^R|_1$ based on the error equation \eqref{eqn:Error_Eqn}. We have the following theorem, which can be viewed as the 2L-ROM analogue of the finite element results given in \cite[Theorem 2.1]{layton1993two}. We will place ourselves in the setting that both the steady NSE \eqref{eqn:NSE_Weak_Formulation} and the 2L-ROM system \eqref{eqn:Alg_Nonlin_2L}--\eqref{eqn:Alg_lin_2L} have a unique solution, denoted by $u$ and $(u^r, u^R)$, respectively. Note that based on the well-posedness results presented in Section~\ref{Sec_wellposdeness}, it suffices to assume the condition given by \eqref{eqn:prop} holds; that is, $Re^2N|f|_* < 1$ 
(see Remark~\ref{Rmk:uniquess_cond}). Let us also define 
\begin{equation} \label{Eq_alpha}
\alpha \coloneqq 1 - Re^2N_{R,r}|f|_*.
\end{equation}
Recall that $N_{R,r} \le N$ as pointed out in Remark~\ref{Rmk:uniquess_cond}. We have thus $\alpha > 0$ under the condition \eqref{eqn:prop}.

\begin{theorem}[Two-Level ROM Error Bound]
\label{thm:main_them}

Let $\Omega$ and $f$ be as given in Proposition~{\em \ref{prop:2.1}}. Assume furthermore that $\partial\Omega$ is $C^2$ smooth and that \eqref{eqn:prop} holds. Let $\alpha$ be defined by \eqref{Eq_alpha}. Then the error $u - u^R$ satisfies
\begin{align} \label{Eq_err_est_goal}
    \begin{split}
    |u - u^R|_1 & \leq \inf_{v^R \in X^R} \Big\{(\alpha^{-1} + 1)|u-v^R|_1 \\
    & \qquad + \alpha^{-1}Re C(\Omega)\|u-v^R\|_0^{1/2} |u-v^R|_1^{1/2}|u^r|_1  \\
    & \qquad + \alpha^{-1}Re C(\Omega) \|u^r\|^{1/2}_0 |u^r|^{1/2}_1 |u-v^R|_1 \Big\} \\
    &\qquad  + \alpha^{-1}Re \inf_{\chi^h \in Q^h} \|p-\chi^h\|_0 \\
    & \qquad + \alpha^{-1}Re C(\Omega)\|u-u^r\|_0^{1/2}|u-u^r|_1^{3/2},
        \end{split}
\end{align}
where $C(\Omega) > 0$ is a generic constant depending only on the spatial domain $\Omega$.
\end{theorem}

\begin{proof}
    We start with the error equation, (\ref{eqn:Error_Eqn}), and set $v^R = \phi^R$ to get
    \begin{align}
    \label{eqn:Error_Eqn_phis}
    \begin{split}
        a(\phi^R,\phi^R) & + b_{SKEW}(\phi^R;u^r,\phi^R) + b_{SKEW}(u^r;\phi^R,\phi^R) \\
       & = a(\eta^R,\phi^R) + b_{SKEW}(\eta^R;u^r,\phi^R) + b_{SKEW}(u^r;\eta^R,\phi^R) \\
        & \hspace{5.5em} + b_{SKEW}(u-u^r;u-u^r,\phi^R) + c(p - \chi^h,\phi^R).
    \end{split}
    \end{align}
For the terms on the LHS of \eqref{eqn:Error_Eqn_phis}, we have $a(\phi^R,\phi^R) = Re^{-1}|\phi|_1^2$ by definition, $b_{SKEW}(u^r;\phi^R,\phi^R) = 0$ by skew symmetry, and 
\begin{equation}
|b_{SKEW}(\phi^R;u^r,\phi^R)| \leq N_{R,r} |\phi^R|_1^2|u_r|_1,
\end{equation}
where $N_{R,r}$ is defined in \eqref{Eq_def_constants}. As a result, the LHS of \eqref{eqn:Error_Eqn_phis} can be bounded from below as follows:  
\begin{equation} \label{Eq_LHS_est}
\begin{aligned}
a(\phi^R,\phi^R) + b_{SKEW}(\phi^R;u^r,\phi^R) & + b_{SKEW}(u^r;\phi^R,\phi^R)  \\
&\quad \ge Re^{-1}|\phi|_1^2 - N_{R,r}|\phi^R|_1^2|u_r|_1.
\end{aligned}
\end{equation}

Next, we derive an upper bound for the terms on the right-hand side (RHS) of \eqref{eqn:Error_Eqn_phis}. First note that
\begin{equation} \label{Eq_easy_terms_est}
\begin{aligned}
    a(\eta^R,\phi^R) & = Re^{-1}(\nabla\eta^R,\nabla\phi^R) \leq Re^{-1}|\eta^R|_1|\phi^R|_1, \\
    c(p-\chi^h,\phi^R) & = - \int_\Omega(\nabla\cdot\phi^R)(p-\chi^h)\,dx \leq |\phi^R|_1\|p-\chi^h\|_0.
\end{aligned}
\end{equation}

In order to bound the remaining three $b_{SKEW}$ terms, we will make use of the following 
inequality for the trilinear term $b$ defined in \eqref{Def_b_term}: 
\begin{equation} \label{Eq_b_est}
|b(u;v,w)| \le C(\Omega) \sqrt{\|u\|_0 |u|_1}\, |v|_1 |w|_1, \quad \text{ for all }  u, v, w \in X,
\end{equation}
which holds provided that the domain $\Omega \subset \mathbb{R}^d$ ($d = 2, 3$) is bounded and has $C^2$ boundary, where $C(\Omega)$ is a generic constant depending only on $\Omega$. See e.g.~\cite[Lemma 22]{layton2008introduction} and \cite[Lemma 61.1]{sell2013dynamics}.  

Note that the same estimate given in \eqref{Eq_b_est} holds for $b_{SKEW}$ as well by simply inspecting the definition of $b_{SKEW}$ given by \eqref{Def_b_skew_term}. Namely, the following estimate holds under our assumption: 
\begin{equation} \label{Eq_bskew_est_generic}
|b_{SKEW}(u;v,w)| \le C(\Omega) \sqrt{\|u\|_0 |u|_1}\, |v|_1 |w|_1, \quad \text{ for all }  u, v, w \in X.
\end{equation}

Applying \eqref{Eq_bskew_est_generic} to the three $b_{SKEW}$ terms on the RHS of \eqref{eqn:Error_Eqn_phis}, we get
\begin{equation} \label{Eq_est_b_skews}
\begin{aligned}
 |b_{SKEW}(\eta^R;u^r,\phi^R)|  & \le C(\Omega) \|\eta^R\|^{1/2}_0 |\eta^R|^{1/2}_1 |u^r|_1 |\phi^R|_1, \\
 |b_{SKEW}(u^r;\eta^R,\phi^R)| & \le C(\Omega) \|u^r\|^{1/2}_0 |u^r|^{1/2}_1 |\eta^R|_1 |\phi^R|_1, \\
 |b_{SKEW}(u-u^r;u-u^r,\phi^R)| & \le C(\Omega) \|u-u^r\|^{1/2}_0 |u-u^r|^{3/2}_1 |\phi^R|_1.
\end{aligned}
\end{equation}

By using the estimates \eqref{Eq_easy_terms_est} and \eqref{Eq_est_b_skews} for the terms on the RHS of \eqref{eqn:Error_Eqn_phis} as well as the lower bound derived in \eqref{Eq_LHS_est} for the LHS of \eqref{eqn:Error_Eqn_phis}, we get, after canceling one factor of $|\phi^R|_1$ on both sides, that
\begin{equation} \label{eqn:sim_to_2.7b}
\begin{aligned}
(Re^{-1} - & N_{R,r}|u^r|_1)|\phi^R|_1 \leq Re^{-1}|\eta^R|_1 + \|p-\chi^h\|_0 \\
& \quad + C(\Omega) \|\eta^R\|^{1/2}_0 |\eta^R|^{1/2}_1 |u^r|_1 + C(\Omega) \|u^r\|^{1/2}_0 |u^r|^{1/2}_1 |\eta^R|_1\\
& \quad + C(\Omega) \|u-u^r\|^{1/2}_0 |u-u^r|^{3/2}_1.
\end{aligned}
\end{equation}

Recall from \eqref{Eq_new_err_terms} that $\eta^R = u-w^R$, and $\phi^R = u^R - w^R$. By substituting these into \eqref{eqn:sim_to_2.7b},  we have
\begin{align}
    \label{eqn:2.8}
    \begin{split}
        (Re^{-1}-N_{R,r}|u^r|_1) |u^R-w^R|_1 & \leq Re^{-1}|u-w^R|_1 + \|p-\chi^h\|_0 \\
        & \quad + C(\Omega)\|u-w^R\|_0^{1/2} |u-w^R|_1^{1/2}|u^r|_1 \\
        & \quad + C(\Omega) \|u^r\|^{1/2}_0 |u^r|^{1/2}_1 |u-w^R|_1 \\
        & \quad + C(\Omega)\|u-u^r\|_0^{1/2}|u-u^r|_1^{3/2}.
    \end{split}
\end{align}
Since the quantity $\alpha = 1 - Re^2N_{R,r}|f|_*$ defined by \eqref{Eq_alpha} is positive thanks to the condition \eqref{eqn:prop},  
after multiplying \eqref{eqn:2.8} by $Re$ and dividing both sides by $\alpha$, we obtain
\begin{align}
    \label{eqn:2.10}
    \begin{split}
        |u^R-w^R|_1 & \leq \alpha^{-1}|u-w^R|_1 + \alpha^{-1}Re\|p-\chi^h\|_0 \\
        & \quad + \alpha^{-1}Re C(\Omega)\|u-w^R\|_0^{1/2} |u-w^R|_1^{1/2}|u^r|_1 \\
        & \quad + \alpha^{-1}Re C(\Omega) \|u^r\|^{1/2}_0 |u^r|^{1/2}_1 |u-w^R|_1 \\
        & \quad + \alpha^{-1}Re C(\Omega)\|u-u^r\|_0^{1/2}|u-u^r|_1^{3/2}.
    \end{split}
\end{align}
Now by adding $|u-w^R|_1$ to both sides of \eqref{eqn:2.10} and applying the triangle inequality to the resulting LHS, we get
\begin{align*}
    |u - u^R|_1 & \leq (\alpha^{-1} + 1)|u-w^R|_1 + \alpha^{-1}Re\|p-\chi^h\|_0 \\
    & \quad + \alpha^{-1}Re C(\Omega)\|u-w^R\|_0^{1/2} |u-w^R|_1^{1/2}|u^r|_1 \\
        & \quad + \alpha^{-1}Re C(\Omega) \|u^r\|^{1/2}_0 |u^r|^{1/2}_1 |u-w^R|_1 \\
        & \quad + \alpha^{-1}Re C(\Omega)\|u-u^r\|_0^{1/2}|u-u^r|_1^{3/2}.
\end{align*}
In the above inequality, since we are free to choose $w^R \in X^R$, we can take the infimum over $X^R$ for all the terms on the RHS that involve $w^R$. Likewise, we can take the infimum over $\chi^h \in Q^h$ for the $\|p-\chi^h\|_0$ term. We have thus 
\begin{align*}
    \begin{split}
        |u - u^R|_1 & \leq \inf_{w^R \in X^R} \big\{(\alpha^{-1} + 1)|u-w^R|_1  \\
 	& \quad + \alpha^{-1}Re C(\Omega)\|u-w^R\|_0^{1/2} |u-w^R|_1^{1/2}|u^r|_1 \\
        & \quad + \alpha^{-1}Re C(\Omega) \|u^r\|^{1/2}_0 |u^r|^{1/2}_1 |u-w^R|_1 \big\} \\
        & \quad +  \alpha^{-1}Re \inf_{\chi^h \in Q^h} \|p-\chi^h\|_0  \\        
        &\quad  + \alpha^{-1}Re C(\Omega)\|u-u^r\|_0^{1/2}|u-u^r|_1^{3/2}.
    \end{split}
\end{align*}
This is exactly the desired result given in \eqref{Eq_err_est_goal}.

\end{proof}

\section{Numerical Results}
    \label{sec:numerical-results}

In this section, we compare the new 2L-ROM to the standard 1L-ROM in the numerical simulation of the steady Burgers equation.
In Section~\ref{sec:numerical-results-mathematical-model}, we present the mathematical model used in our numerical investigation.
In Section~\ref{sec:numerical-results-rom}, we present the development of the 2L-ROM and 1L-ROM, including the ROM basis and operator construction, the Newton solver initialization, and the criteria used in our numerical investigation.
Finally, in Section~\ref{sec:numerical-results-results}, we present and discuss results for two numerical experiments.

\subsection{Mathematical Model}
    \label{sec:numerical-results-mathematical-model}

For simplicity, instead of the steady NSE used to construct and analyze the 2L-ROM in Sections~\ref{sec:preliminaries}--\ref{sec:theoretical-results}, in our numerical investigation we consider 
the steady viscous Burgers equation:
\begin{align}
    \label{Steady_Burgers}
    \begin{split}
        - \nu u_{xx} + uu_x  & = f(x), \quad x \in (a,b) \\
        u(a) & = \alpha, \quad u(b) = \beta,
    \end{split}
\end{align}
where $a, b, \alpha$, and $\beta$ are parameters, $\Omega = [a,b]$ is the computational domain, $\nu$ is the diffusion coefficient, and $f$ is the forcing term.
In our numerical investigation, we consider the method of manufactured solutions with the following exact solution:
\begin{align}
\label{eqn:manu_soln}
    u(x) = 1 - \frac{2(x-a)}{b-a} + \frac{\exp{\big(-\frac{(x-q)^2}{2\sigma^2}\big)}\sin{\big(\frac{k\pi(x-a)}{(b-a)}}\big)}{\sqrt{2\pi}\sigma 
    },
\end{align}
where $q, \sigma$, and $k$ are parameters.
The corresponding forcing term is obtained by plugging in the exact solution~\eqref{eqn:manu_soln} into the Burgers equation~\eqref{Steady_Burgers}.

The exact solution~\eqref{eqn:manu_soln} consists of a linear component and the product between a Gaussian profile and a sine function.
Graphically, the exact solution can be described as a moving wave along the line 
$y = 1 - 2(x-a)/(b-a)$, 
which propagates as the parameter $q$ varies (see Fig.~\ref{fig:exact_sol_8qs}).
The parameter $k$ is the wave number in the sine 
component of the exact solution. 
The 
parameter $\sigma$ is the standard deviation of the Gaussian 
component of the exact solution.
By varying the parameter $q$, we obtain different profiles in the exact solution~\eqref{eqn:manu_soln}, which we use as snapshots in the ROM basis construction.

In our numerical investigation, we use the following parameter values:
We set the computational domain $\Omega = [-4,4]$, i.e., $a = -4$ and $b = 4$ in~\eqref{Steady_Burgers}.
We choose the boundary conditions $u(-4) = 1, u(4) = -1$, i.e., $\alpha = 1$ and $\beta = -1$ in~\eqref{Steady_Burgers}.
We also set  $k = 1$ and $\sigma = 0.5$ in the exact solution~\eqref{eqn:manu_soln}.
To generate the snapshots, in the exact solution~\eqref{eqn:manu_soln} we use $q \in [-4,4]$ with increments of $1/100$, which yields $801$ snapshots. 
To illustrate the qualitative differences in the snapshots, in Figure~\ref{fig:exact_sol_8qs} we plot several representative exact solutions~\eqref{eqn:manu_soln} for different $q$ values.

\begin{figure}[!htb]
    \centering
    \includegraphics[scale=0.5]{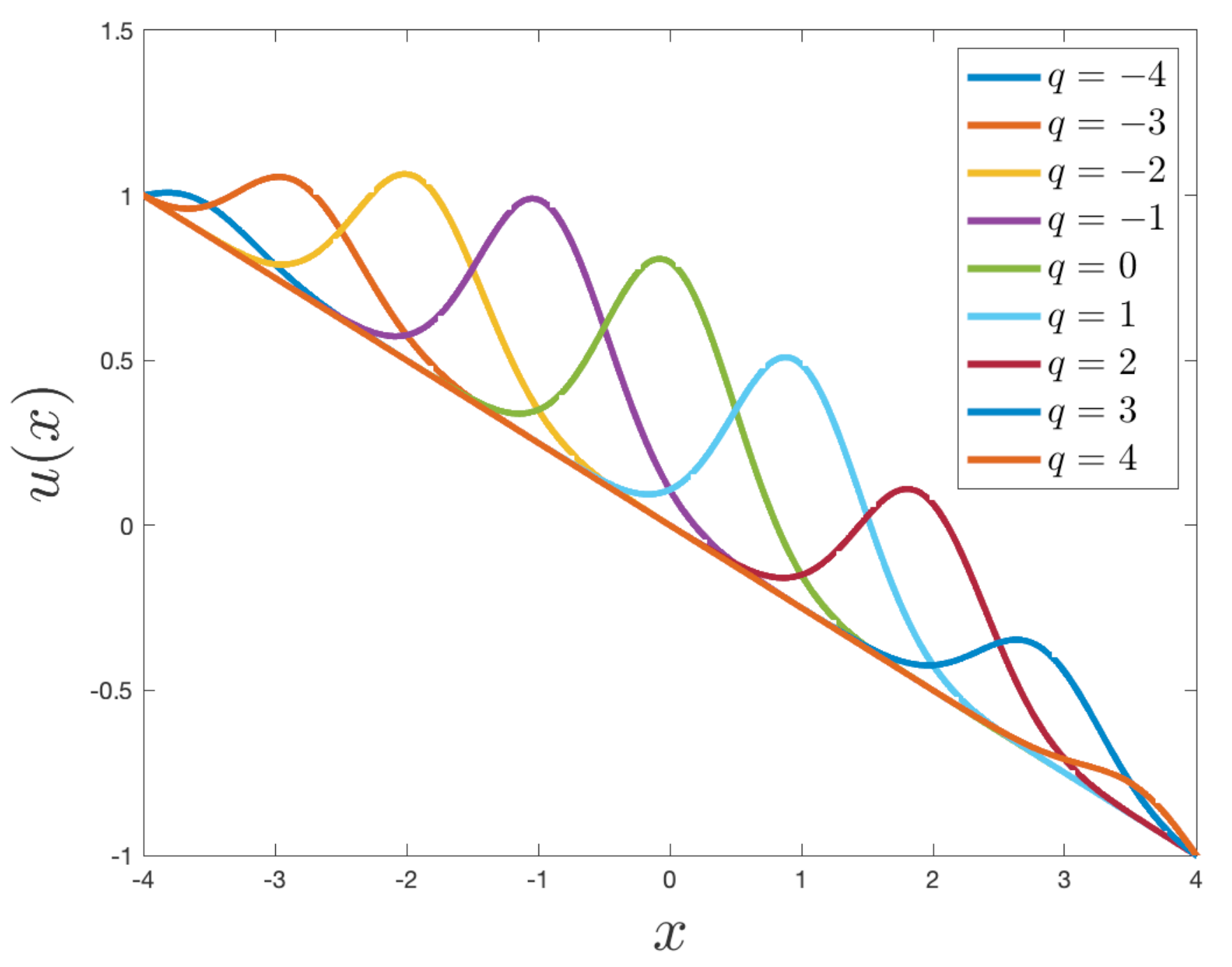}
    \caption{
    	Exact solutions 
    	\eqref{eqn:manu_soln} of the steady viscous Burgers problem \eqref{Steady_Burgers} 
    	for varying $q$ values.
    	\label{fig:exact_sol_8qs}
    }
\end{figure}

\subsection{Reduced Order Model}
    \label{sec:numerical-results-rom}

\subsubsection{Basis Construction}

To generate the POD basis, we first take the $801$ exact solutions that were obtained by varying $q$, and subtract the corresponding average true solutions, $u_{mean}$, for all $q$ parameters.
This lifting procedure ensures that the resulting snapshots (and the POD basis functions obtained from them) satisfy homogeneous Dirichlet boundary conditions~\cite{volkwein2013proper}.
Next, for practical convenience, we interpolate these lifted snapshots onto a quadratic finite element mesh of mesh size $h = 1/200$ over the domain $\Omega = [-4,4]$, which yields snapshot vectors of dimension $3201$.
The rank of the resulting snapshot matrix is $30$.
Finally, we use the POD algorithm~\cite{volkwein2013proper} to generate the POD basis $\{\varphi_1, \varphi_2, \ldots, \varphi_{30}\}$.
For illustrative purposes, in Figure~\ref{fig:POD_Basis_Funcs} we plot the POD basis functions $\varphi_{1}$ and $\varphi_{30}$.

\begin{figure}[!htb]
    \centering
    \includegraphics[scale=0.5]{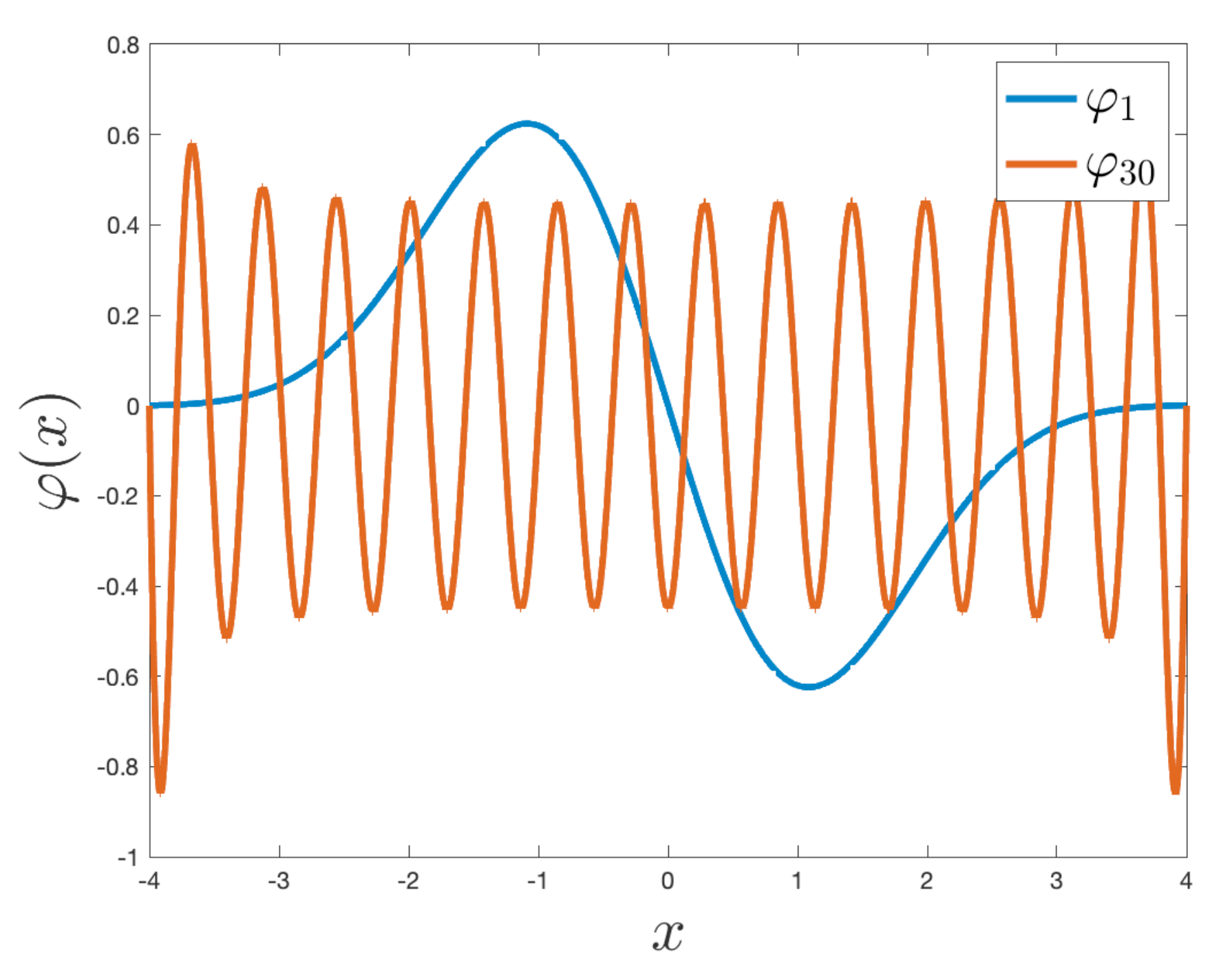}
    \caption{
    	POD basis functions $\varphi_{1}$ and $\varphi_{30}$ 
    	constructed from the manufactured solutions \eqref{eqn:manu_soln} as $q$ is varied. 
    	\label{fig:POD_Basis_Funcs}
    }
\end{figure}
 
 \subsubsection{ROM System}

In this section, we present the construction of the ROM operators for the new 2L-ROM, outlined in Algorithm \ref{alg:two_level_ROM}.
In Algorithm~\ref{alg:burgers_two_level_ROM}, we outline the construction of the 2L-ROM operators for the steady Burgers equation, which is the mathematical model used in the numerical investigation of the 2L-ROM and 1L-ROM in Section~\ref{sec:numerical-results}.

 \begin{algorithm}[H]
\caption{Steady Burgers Two-Level ROM Algorithm}
\label{alg:burgers_two_level_ROM}
Step 1. 
Use Newton's method to solve the following nonlinear, 
$r$-dimensional problem, with $r$ POD basis
functions, $\{\varphi_1,\varphi_2,\ldots, \varphi_r\}$: 
\begin{align}
    \label{eqn:burgers_Alg_Nonlin_2L}
    A^ra^r + (a^{r})^TB^ra^r + b^r = 0, 
\end{align}
where $a^r$ is the $r \times 1$ vector of unknown ROM coefficients (corresponding to $u^r \in X^r$), $A^r$ is an $r \times r$ matrix, $B^r$ is an $r \times r \times r$ tensor, and $b^r$ is a $r \times 1$ vector.
The ROM operators in~\eqref{eqn:burgers_Alg_Nonlin_2L} are given by the following formulas: 
For $1 \leq i,j,k \leq r$,
\begin{align*}
    A_{ij}^r & = 
    \nu \, (\nabla \varphi_j , \nabla \varphi_i), \\
    B_{ijk}^r & = 
    (\varphi_j(\varphi_k)_x,\varphi_i), \\
    b_{i}^r & = - (f,\varphi_i).
\end{align*}
Step 2. Solve the following linear, 
$R$-dimensional problem, with $R$ POD basis functions, $\{\varphi_1, \varphi_2, \ldots, \varphi_R\}$: 
\begin{align}
    \label{eqn:burgers_Alg_lin_2L}
    A^R a^R 
    + (\tilde{a}^r)^T B^R a^R 
    + (a^R)^T B^R \tilde{a}^r 
    -(\tilde{a}^r)^T B^R \tilde{a}^r 
    + b^R
    = 0.
\end{align}
where $\tilde{a}^r = [a^r,0,\ldots,0]^T$ is an $R \times 1$ vector constructed by using the vector $a^r$ from Step 1, $a^R$ is the $R \times 1$ vector of unknown ROM coefficients (corresponding to $u^R \in X^R$), $A^R$ is an $R \times R$ matrix, $B^R$ is an $R \times R \times R$ tensor, and $b^R$ is an $R \times 1$ vector.
The ROM operators in~\eqref{eqn:burgers_Alg_Nonlin_2L} are given by the following formulas: 
For $1 \leq i,j,k \leq R$,
\begin{align*}
    A_{ij}^R & = 
    \nu \, (\nabla \varphi_j , \nabla \varphi_i), \\
    B_{ijk}^R & = (\varphi_j(\varphi_k)_x,\varphi_i), \\
    b_{i}^R & = - (f,\varphi_i).
\end{align*}
\end{algorithm}

\subsubsection{Newton Solver Initial Guesses} 

To solve the nonlinear problems in the 2L-ROM and 1L-ROM, we use a Newton solver with a residual and step tolerance of $1\mathrm{E-}10$. In order to accurately depict numerical results, we use three different initial guesses, $u_0$.  

The first initial guess we consider, $u_{UG}$, 
is called the uninformed guess. The uninformed guess 
represents a blind but poor initial guess that takes the largest number of Newton steps in order to converge to the solution. 
This case is most similar to more difficult problems that require many Newton iterations to solve. 

The second initial guess we consider, 
$u_{IG}$, 
is called the informed guess. The informed guess replicates a guess that is based on a loose understanding of what the exact solution approximately is, and so we will still have to take several Newton steps, but not as many as in the uninformed guess. 

The third initial guess, $u_{avg}$, is called the best guess. 
In this case, we already know 
what the exact solution is, and thus this guess will require the fewest Newton iterations. 
The best guess, 
$u_{avg}$, is equal to the average of all of the exact solutions that were used to generate the POD basis. 

For the manufactured solution~\eqref{eqn:manu_soln}, we 
define the initial guesses 
$u_{UG}$, $u_{IG}$, and $u_{avg}$ by their corresponding ROM coefficient vectors $a_{UG}$, $a_{IG}$, and $a_{avg}$, respectively, as follows: 
\begin{align}
    \label{eqn:uninformed_guess}
    a_{UG} & = 
    \begin{cases}
    [1,-1,1,\ldots,1,-1,0,0]^T \quad \text{if $r$ is even} \\
    [1,-1,1,\ldots,-1,1,0,0]^T \quad \text{if $r$ is odd} \\
    \end{cases} \\
    a_{IG} & = 
    \frac{1}{2}a_{UG} \\
    a_{avg} & = [0,\ldots, 0]^T.
\end{align}

The choice of $u_{UG}$ was 
made to push the Newton solver to the limit. 
The vectors $[1,-1,1,\ldots]^T$ and $[1,-1,1,\ldots,0]^T$ both yielded diverging Newton iterations for some parameter values $p$ when paired with specific $r$ values of interest. Thus, we ended up with $u_{UG}$ in (\ref{eqn:uninformed_guess}), because the Newton solver converges for all parameter $p$ values for all $r$ values of interest. Note that $u_{avg}= 0 $ because we lifted our exact solution before we found the POD basis. 
Notice that $u_{IG}$ is exactly halfway between $u_{UG}$ and $u_{avg}$ both in the ROM space and the when projected and lifted back into the full order space. 
We 
display the different 
initial guesses 
in Figure \ref{fig:ROM_init_guess} for $r = 16$ and $r = 25$. 

\begin{figure}
\begin{subfigure}[h]{0.5\linewidth}
\includegraphics[width=\linewidth]{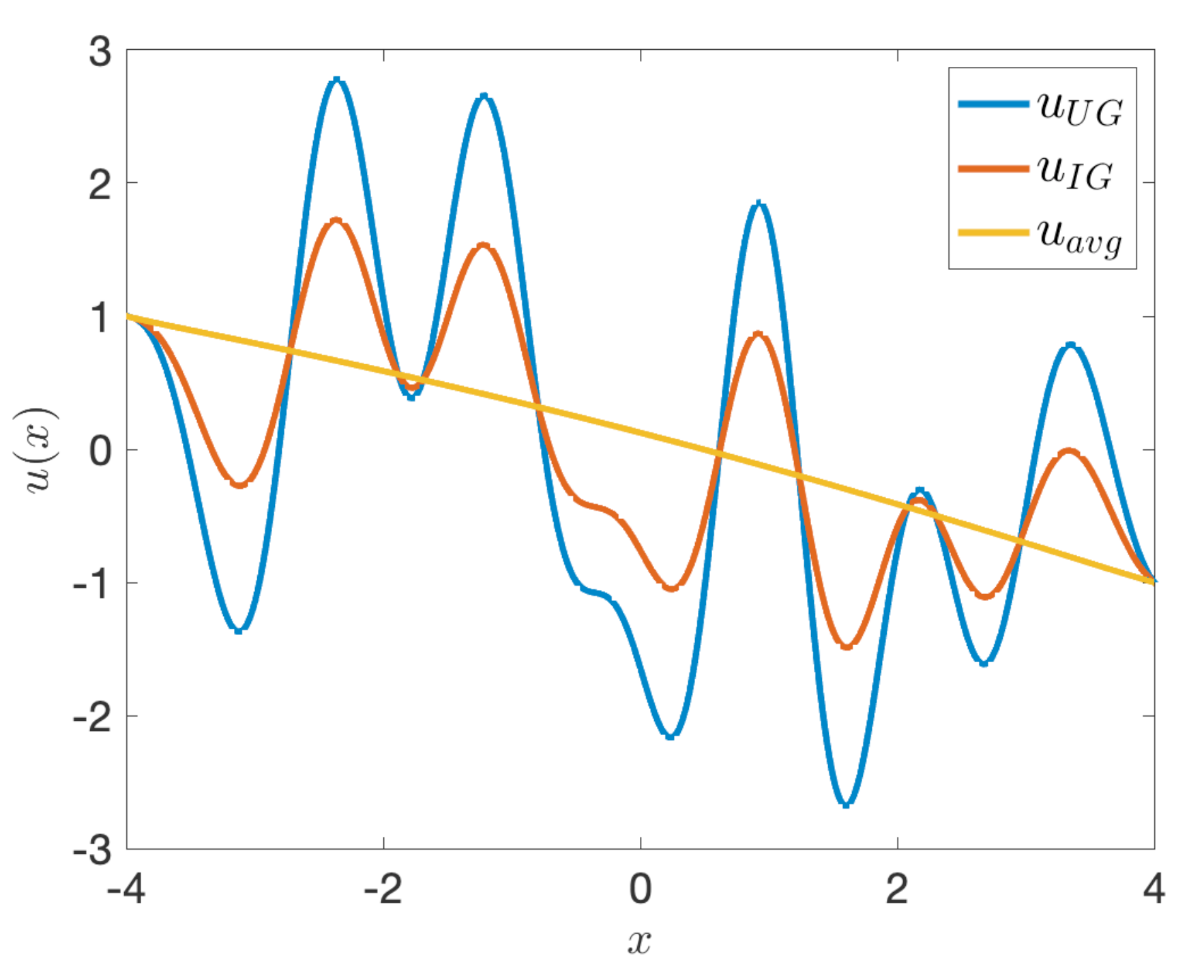}
\caption{$r = 16$}
\end{subfigure}
\hfill
\begin{subfigure}[h]{0.5\linewidth}
\includegraphics[width=\linewidth]{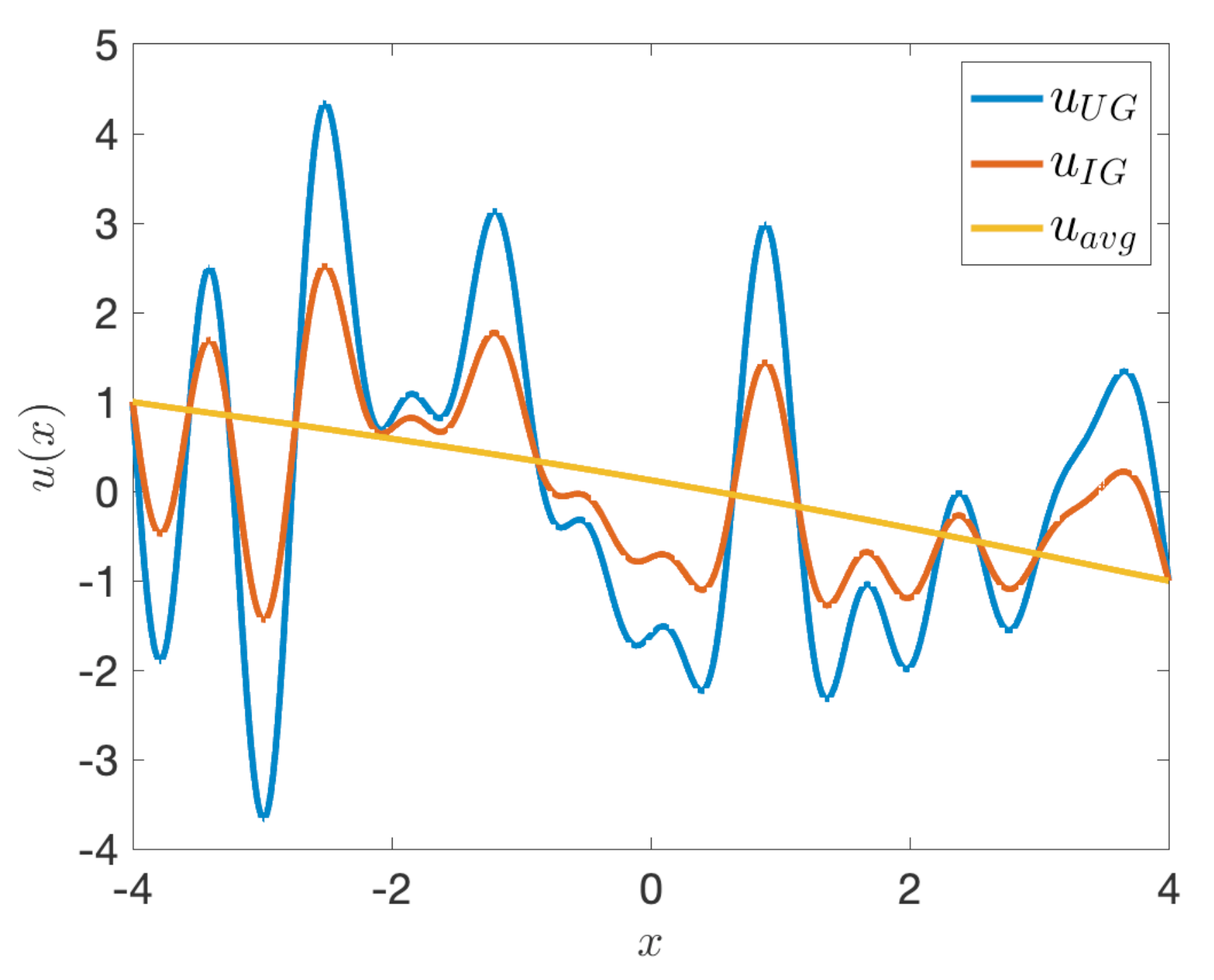}
\caption{$r = 25$}
\end{subfigure}%
\caption{ROM initial guesses for two $r$ values.}
\label{fig:ROM_init_guess}
\end{figure}

\subsubsection{Criteria for the Investigation}
    \label{sec:numerical-results-criteria}
    
All models were run on a Macbook pro with a 6-core processor. We ran each model $100$ times and averaged the time it took to complete those 
simulations. 

We compare the 2L-ROM to the 
1L-ROM. 
The 2L-ROM uses $r$ modes in the nonlinear solver and $R_2$ modes (with $r < R_2$) in the linear solver, whereas the 1L-ROM uses $R_1$ modes in the nonlinear solver. 
For the 2L-ROM we say we are using an $(r,R_2)$ pair to designate how many modes are being used for the nonlinear solver and the linear solver, respectively. 
Since we have three free parameters (i.e., $r, R_1$, and $R_2$), to simplify the presentation, we consider the following two experiments: 
\begin{itemize}
    \item {\bf Experiment 1}: \ 
    $r < R_1 = R_2$.
    In this experiment, we set the number of modes used in the linear solver of the 2L-ROM equal to the number of modes used in the nonlinear solver of the 1L-ROM. 
    \item {\bf Experiment 2}: \ 
    $r < R_1 < R_2$.
    In this experiment, we choose the number of modes used in the linear solver of the 2L-ROM to be greater than the number of modes used in the nonlinear solver of the 1L-ROM. 
    We choose 
    $R_1 < R_2$ because this choice increases the accuracy of the 2L-ROM for a small increase in computational cost.
\end{itemize}

We note that the 2L-ROM settings in both Experiment 1 and Experiment 2 are covered by the theoretical results in Theorem~\ref{thm:main_them}.

\subsection{
Results 
and Discussion}
    \label{sec:numerical-results-results}

In this section, we present numerical results for the new 2L-ROM and the standard 1L-ROM for the two experiments described in Section~\ref{sec:numerical-results-criteria}.
In our numerical comparison of the 2L-ROM and 1L-ROM, we focus on two criteria: numerical accuracy and computational cost.
In particular, we investigate whether the new 2L-ROM can significantly reduce the computational cost of the standard 1L-ROM without reducing its accuracy.

\subsubsection{Experiment 1 \ ($r < R_1 = R_2$)}

For clarity, in this section we use the notation $R \coloneqq R_1 = R_2$.
In Table \ref{table:Manu_Soln_L2_Errors_Times_Final}, we list the 
errors and timings 
for the 2L-ROM and 
1L-ROM for the $(r,R)$ pairings $(12,23), (14,25), (16,25), (18,25), (20,29)$ in Table \ref{table:Manu_Soln_L2_Errors_Times_Final}. 
We note that, in Table \ref{table:Manu_Soln_L2_Errors_Times_Final}, we average the errors and timings for all parameter $p$ values. In the first column of the table, we list the initial guesses, $u_{0}$, used for each model. In the next four columns, we list the error of the 2L-ROM, $\cE(2L)$, the average time in seconds it took to run the 2L-ROM, $2L$ Time, the error of the 1L-ROM, $\cE(1L)$, and the average time in seconds it took to run the 1L-ROM, $1L$ Time. 
We then list the ``Error Ratio" column, which is calculated by evaluating the ratio  $\cE(2L)/\cE(1L)$ from the same row. In the last column, we list the ``Speedup," which is calculated by evaluating the ratio of the 1L-ROM average time 
to the 2L-ROM average time.

Overall, the results in Table~\ref{table:Manu_Soln_L2_Errors_Times_Final} yield the following two conclusions regarding the new 2L-ROM and the standard 1L-ROM:
\begin{itemize}
    \item The 2L-ROM errors are similar to the 1L-ROM errors.
    \item The 2L-ROM computational cost is significantly (sometimes by a factor of $2$ or $3$) lower than the 1L-ROM computational cost.
\end{itemize}
We emphasize, however, that these 
two conclusions are conditional on whether the $(r,R)$ pair is chosen well. If we choose $r$ to be too small in comparison to $R$, we ensure computational efficiency but at the expense of the error accuracy. 
On the other hand, if we choose $r$ to be too close to $R$ the errors of the 2-ROM and 1L-ROM are nearly identical, but then the speedup will be close to $1$. 

We note that the choice of the initial guess in the Newton solver 
does not affect the accuracy of the 2L-ROM and 1L-ROM. 
We also not that, as $r$ approaches $R$, the error ratio between the 1L-ROM and 2L-ROM approaches $1$. 
This behavior is natural since we 
do not expect the 2L-ROM with the $(r,R)$ pair to 
be more accurate than the 1L-ROM 
with $R$ modes, because 
solving the $R$-dimensional nonlinear problem should yield more accurate results than solving the $R$-dimensional linearization.
We also note that, to 
achieve a 2L-ROM error similar to the 1L-ROM, 
$r$ does not have to be almost equal to $R$. 
For example, 
in the $(16,25)$ case, 
the 2L-ROM error is less than $3\%$ larger than the 1L-ROM error.
We 
also note that 
the error ratio can either go up or down depending on our choices of $r$ and $R$. 
For example, the pair $(12,23)$ produces a $27.5\%$ difference in error, but $(18,25)$ produces less than a $1\%$ difference in error. Thus, 
even for $r$ being approximately half the value of $R$, 
the 2L-ROM error can be similar to the 1L-ROM error.

Next, we investigate the computational efficiency of the 2L-ROM and 1L-ROM.
First, we note that the choice of the initial guess in the Newton solver has a significant influence on the speedup.
The reason is that the choice of the initial guess affects how many Newton steps are needed to converge to the solution. 
The more Newton iterations are needed, the better the 2L-ROM 
performs in comparison to the 1L-ROM. 
This is clearly seen for $u_0 = u_{UG}$ for all $(r,R)$ pairs in Table~\ref{table:Manu_Soln_L2_Errors_Times_Final}. The speedup ranges from $2$ to $3$ depending on the $(r,R)$ pair. 
Specifically, 
for the $(16,25)$ pair 
and $u_0 = u_{UG}$, the error 
difference is less than $3\%$ 
but 
the speedup is 
$2.966$. 
We also note that the better 
the initial guess becomes the faster the nonlinear solver converges, and thus the benefit of the 2L-ROM is lessened. 
When using an initial guess halfway between our unintuitive guess and optimal guess, $u_0 = u_{IG}$, 
the speedup 
is about $2$ while 
the error ratio is close to $1$. 
Even when using an optimal initial guess, $u_0 = u_{avg}$, the speedup is above $1$ for reasonably chosen $(r,R)$ pairs. 
Overall, the results in Table~\ref{table:Manu_Soln_L2_Errors_Times_Final} show that, for well chosen $(r,R)$ pairs, the new 2L-ROM can reduce the computational cost of the standard 1L-ROM by a factor of $2$ and even $3$, without significantly increasing the error.


\begin{table}[ht]
    \centering
    \begin{tabular}{|c|c|c|c|c|c|c|}
\hline													
\multicolumn{7}{|c|}{$L^2$ errors and times: averaged over all $q$ values} \\												
    \hline													
    \multicolumn{7}{|c|}{$(12,23)$} \\													
    \hline													
    $u_0$	&	$\cE(2L)$	&	$2L$ Time (s)	&	$\cE(1L)$	&	$1L$ Time (s)	&	Error Ratio	&	Speedup	\\
    \hline													
    $u_{UG}$	&	2.805E-04	&	7.425E-04	&	2.199E-04	&	1.958E-03	&	1.275	&	2.637	\\
    $u_{IG}$	&	2.805E-04	&	6.642E-04	&	2.199E-04	&	1.337E-03	&	1.275	&	2.012	\\
    $u_{avg}$	&	2.805E-04	&	6.054E-04	&	2.199E-04	&	9.579E-04	&	1.275	&	1.582	\\
    \hline													
    \multicolumn{7}{|c|}{$(14,25)$} \\	
    \hline													
    $u_0$	&	$\cE(2L)$	&	$2L$ Time (s)	&	$\cE(1L)$	&	$1L$ Time (s)	&	Error Ratio	&	Speedup	\\
    \hline													
    $u_{UG}$	&	6.225E-05	&	1.003E-03	&	4.714E-05	&	3.049E-03	&	1.321	&	3.041	\\
    $u_{IG}$	&	6.225E-05	&	8.054E-04	&	4.714E-05	&	1.818E-03	&	1.321	&	2.257	\\
    $u_{avg}$	&	6.225E-05	&	7.170E-04	&	4.714E-05	&	1.092E-03	&	1.321	&	1.523	\\
    \hline													
    \multicolumn{7}{|c|}{$(16,25)$} \\													
    \hline													
    $u_0$	&	$\cE(2L)$	&	$2L$ Time (s)	&	$\cE(1L)$	&	$1L$ Time (s)	&	Error Ratio	&	Speedup	\\
    \hline													
    $u_{UG}$	&	4.824E-05	&	1.013E-03	&	4.714E-05	&	3.006E-03	&	1.023	&	2.966	\\
    $u_{IG}$	&	4.824E-05	&	8.965E-04	&	4.714E-05	&	1.823E-03	&	1.023	&	2.034	\\
    $u_{avg}$	&	4.824E-05	&	8.056E-04	&	4.714E-05	&	1.101E-03	&	1.023	&	1.367	\\
    \hline													
    \multicolumn{7}{|c|}{$(18,25)$} \\													
    \hline													
    $u_0$	&	$\cE(2L)$	&	$2L$ Time (s)	&	$\cE(1L)$	&	$1L$ Time (s)	&	Error Ratio	&	Speedup	\\
    \hline													
    $u_{UG}$	&	4.719E-05	&	1.144E-03	&	4.714E-05	&	3.041E-03	&	1.001	&	2.657	\\
    $u_{IG}$	&	4.719E-05	&	1.038E-03	&	4.714E-05	&	1.818E-03	&	1.001	&	1.751	\\
    $u_{avg}$	&	4.719E-05	&	9.178E-04	&	4.714E-05	&	1.088E-03	&	1.001	&	1.186	\\
    \hline	
    \multicolumn{7}{|c|}{$(20,29)$} \\							
    \hline													
    $u_0$	&	$\cE(2L)$	&	$2L$ Time (s)	&	$\cE(1L)$	&	$1L$ Time (s)	&	Error Ratio	&	Speedup	\\
    \hline													
    $u_{UG}$	&	8.815E-07	&	1.596E-03	&	8.475E-07	&	3.309E-03	&	1.04	&	2.074	\\
    $u_{IG}$	&	8.815E-07	&	1.201E-03	&	8.475E-07	&	2.437E-03	&	1.04	&	2.028	\\
    $u_{avg}$	&	8.815E-07	&	1.081E-03	&	8.475E-07	&	1.435E-03	&	1.04	&	1.328	\\
    \hline		
    \end{tabular}
     \caption{Experiment 1: $L^2$ errors and times. 
    }
    \label{table:Manu_Soln_L2_Errors_Times_Final}
\end{table}

\subsubsection{Experiment 2 \ ($r < R_1 < R_2$)}

In Experiment 1 (in the previous section), we  compared the 2L-ROM to the 1L-ROM only for the case $R_2 = R_1$, i.e., when the dimension of the 2L-ROM linearized problem, $R_2$, was equal to the dimension of the 1L-ROM nonlinear problem, $R_1$. 
We emphasize, however, that in practice we could use different $R_2$ and $R_1$ values, which could yield 2L-ROM results that are better than the 1L-ROM results.
To investigate whether this is the case, in Table~\ref{table:Manu_Soln_L2_Errors_Times_Final_r_R1_R2} we list the same information as in Table~\ref{table:Manu_Soln_L2_Errors_Times_Final}, but for the case $r < R_1 < R_2$.
As mentioned in Section~\ref{sec:numerical-results-criteria}, we choose $R_1 < R_2$ to increase the 2L-ROM accuracy for a small increase in computational cost.
We also choose lower $r$ values than those chosen in Experiment 1.
This choice decreased the computational cost of the 2L-ROM. 
Furthermore, the combination $r < R_1 < R_2$ also ensured an overall increase in the 2L-ROM accuracy.  

Overall, the results in Table~\ref{table:Manu_Soln_L2_Errors_Times_Final_r_R1_R2} show that the {\it 2L-ROM can significantly decrease both the 1L-ROM error and the 1L-ROM computational cost}. 
For example, for $(r,R_2) = (20,29)$ and $R_1 = 25$, the 2L-ROM error is $53$ times lower than the 1L-ROM's error and the speedup is larger than $1$ for all three initial guesses.

\begin{table}[ht]
    \centering
    \begin{tabular}{|c|c|c|c|c|c|c|}
    \hline													
    \multicolumn{7}{|c|}{$L^2$ errors and times: averaged over all $q$ values} \\													
    \hline													
    \multicolumn{7}{|c|}{$r < R_1< R_2$, with 2L using $(r,R_2)$ and 1L using $R_1$ } \\													
    \hline													
    &		\multicolumn{2}{|c|}{$(r,R_2) = (10,24)$} &				\multicolumn{2}{|c|}{$R_1 = 21$} 	&	& \\	
    \hline													
    $u_0$	&	$\cE(2L)$	&	$2L$ Time (s)	&	$\cE(1L)$	&	$1L$ Time (s)	&	Error Ratio	&	Speedup	\\
    \hline													
    $u_{UG}$	&	6.352E-04	&	6.571E-04	&	7.420E-04	&	1.494E-03	&	0.856	&	2.274	\\
    $u_{IG}$	&	6.352E-04	&	5.985E-04	&	7.420E-04	&	9.873E-04	&	0.856	&	1.65 	\\
    $u_{avg}$	&	6.352E-04	&	5.418E-04	&	7.420E-04	&	8.244E-04	&	0.856	&	1.522	\\
    \hline													
    &		\multicolumn{2}{|c|}{$(r,R_2) = (11,23)$} 	&			\multicolumn{2}{|c|}{$R_1 = 21$} 				&		&   \\	
    \hline													
    $u_0$	&	$\cE(2L)$	&	$2L$ Time (s)	&	$\cE(1L)$	&	$1L$ Time (s)	&	Error Ratio	&	Speedup	\\
    \hline													
    $u_{UG}$	&	5.400E-04	&	6.636E-04	&	7.420E-04	&	1.494E-03	&	0.728	&	2.251	\\
    $u_{IG}$	&	5.400E-04	&	6.020E-04	&	7.420E-04	&	9.873E-04	&	0.728	&	1.64	\\
    $u_{avg}$	&	5.400E-04	&	5.704E-04	&	7.420E-04	&	8.244E-04	&	0.728	&	1.445	\\
    \hline													
    &		\multicolumn{2}{|c|}{$(r,R_2) = (12,24)$} 	&			\multicolumn{2}{|c|}{$R_1 = 23$} 				&		&	\\
    \hline													
    $u_0$	&	$\cE(2L)$	&	$2L$ Time (s)	&	$\cE(1L)$	&	$1L$ Time (s)	&	Error Ratio	&	Speedup	\\
    \hline													
    $u_{UG}$	&	1.670E-04	&	7.235E-04	&	2.199E-04	&	1.905E-03	&	0.759	&	2.633	\\
    $u_{IG}$	&	1.670E-04	&	6.623E-04	&	2.199E-04	&	1.327E-03	&	0.759	&	2.004	\\
    $u_{avg}$	&	1.670E-04	&	5.897E-04	&	2.199E-04	&	9.491E-04	&	0.759	&	1.609	\\
    \hline													
    &		\multicolumn{2}{|c|}{$(r,R_2) = (14,27)$} 	&			\multicolumn{2}{|c|}{$R_1 = 25$} 				&		&	\\
    \hline													
    $u_0$	&	$\cE(2L)$	&	$2L$ Time (s)	&	$\cE(1L)$	&	$1L$ Time (s)	&	Error Ratio	&	Speedup	\\
    \hline													
    $u_{UG}$	&	3.713E-05	&	7.264E-04	&	4.714E-05	&	2.992E-03	&	0.788	&	4.12	\\
    $u_{IG}$	&	3.713E-05	&	7.267E-04	&	4.714E-05	&	1.795E-03	&	0.788	&	2.469	\\
    $u_{avg}$	&	3.713E-05	&	7.428E-04	&	4.714E-05	&	1.070E-03	&	0.788	&	1.441	\\
    \hline													
    &		\multicolumn{2}{|c|}{$(r,R_2) = (14,28)$} 	&			\multicolumn{2}{|c|}{$R_1 = 25$} 				&		&\\	
    \hline													
    $u_0$	&	$\cE(2L)$	&	$2L$ Time (s)	&	$\cE(1L)$	&	$1L$ Time (s)	&	Error Ratio	&	Speedup	\\
    \hline													
    $u_{UG}$	&	3.634E-05	&	1.003E-03	&	4.714E-05	&	2.992E-03	&	0.771	&	2.983	\\
    $u_{IG}$	&	3.634E-05	&	8.386E-04	&	4.714E-05	&	1.795E-03	&	0.771	&	2.14	\\
    $u_{avg}$	&	3.634E-05	&	7.315E-04	&	4.714E-05	&	1.070E-03	&	0.771	&	1.463	\\
    \hline
    &		\multicolumn{2}{|c|}{$(r,R_2) = (20,29)$} &				\multicolumn{2}{|c|}{$R_1 = 25$} 				&		&	\\
    \hline													
    $u_0$	&	$\cE(2L)$	&	$2L$ Time (s)	&	$\cE(1L)$	&	$1L$ Time (s)	&	Error Ratio	&	Speedup	\\
    \hline													
    $u_{UG}$	&	8.815E-07	&	1.579E-03	&	4.714E-05	&	2.992E-03	&	0.019	&	1.895	\\
    $u_{IG}$	&	8.815E-07	&	1.193E-03	&	4.714E-05	&	1.795E-03	&	0.019	&	1.504	\\
    $u_{avg}$	&	8.815E-07	&	1.049E-03	&	4.714E-05	&	1.070E-03	&	0.019	&	1.02	\\
    \hline
    \end{tabular}
    \caption{
    Experiment 2: $L^2$ errors and times.
    }
    \label{table:Manu_Soln_L2_Errors_Times_Final_r_R1_R2}
\end{table}

\section{Conclusions}
    \label{sec:conclusions}

In this paper, we proposed, analyzed, and investigated numerically a novel two-level G-ROM (2L-ROM) for the efficient and accurate numerical simulation of the steady NSE.
The new 2L-ROM consists of two steps:
In the first step, a relatively low-dimensional G-ROM for the (nonlinear) NSE is solved.
In the second step, the NSE are linearized around the solution found in the first step, and the higher-dimensional G-ROM for the linearized system is solved.
In Theorem~\ref{thm:main_them}, we proved an error bound for the new 2L-ROM.
In Section~\ref{sec:numerical-results}, we compared the new 2L-ROM to the standard 1L-ROM in the numerical simulation of the steady Burgers equation.
Our numerical investigation showed that, as expected, the 2L-ROM could significantly decrease (by a factor of $2$ and even $3$) the 1L-ROM computational cost, without compromising the numerical accuracy.

The first steps in the investigation of the new 2L-ROM are encouraging.
We plan to investigate the 2L-ROM in more challenging numerical settings, e.g., the 2D and 3D NSE.
We also plan to leverage the theoretical error bounds to determine robust scalings for the 2L-ROM parameters $r$ and $R_2$.

\section*{Acknowledgments}

The work of the first, second, and fifth authors was supported by NSF through grant DMS-2012253 and CDS\&E-MSS-1953113. The work of the third author was supported by NSF through grant DMS-2108856. The work of the fourth author was supported through grant DOE ASCR DE-SC0021313.

\bibliographystyle{plain}
\bibliography{traian}

\begin{thebibliography}{10}

\bibitem{ahmed2021closures}
S.~E. Ahmed, S.~Pawar, O.~San, A.~Rasheed, T.~Iliescu, and B.~R. Noack.
\newblock On closures for reduced order models $-$ a spectrum of
  first-principle to machine-learned avenues.
\newblock {\em Phys. Fluids}, 33(9):091301, 2021.

\bibitem{audouze2013nonintrusive}
C.~Audouze, F.~De~Vuyst, and P.~B. Nair.
\newblock Nonintrusive reduced-order modeling of parametrized time-dependent
  partial differential equations.
\newblock {\em Num. Meth. P.D.E.s}, 29(5):1587--1628, 2013.

\bibitem{axe961}
O.~Axelsson and W.~Layton.
\newblock A two-level method for the discretization of nonlinear boundary value
  problems.
\newblock {\em SIAM J. Numer. Anal.}, 33:2359--2374, 1996.

\bibitem{borggaard2008two}
J.~T. Borggaard, T.~Iliescu, H.~Lee, J.~P. Roop, and H.~Son.
\newblock {A Two-Level Smagorinsky Model}.
\newblock {\em Multiscale Modeling and Simulation}, 7(2):599--621, 2008.

\bibitem{chen2018greedy}
W.~Chen, J.~S. Hesthaven, B.~Junqiang, Y.~Qiu, Z.~Yang, and Y.~Tihao.
\newblock Greedy nonintrusive reduced order model for fluid dynamics.
\newblock {\em AIAA J.}, 56(12):4927--4943, 2018.

\bibitem{dai2020two}
X.~Dai, X.~Kuang, J.~Xin, and A.~Zhou.
\newblock Two-grid based adaptive proper orthogonal decomposition method for
  time dependent partial differential equations.
\newblock {\em J. Sci. Comput.}, 84(3):1--27, 2020.

\bibitem{dawson1998two}
C.~N. Dawson, M.~F. Wheeler, and C.~S. Woodward.
\newblock A two-grid finite difference scheme for nonlinear parabolic
  equations.
\newblock {\em SIAM J. Numer. Anal.}, 35(2):435--452, 1998.

\bibitem{fairag1998two}
F.~Fairag.
\newblock A two-level finite-element discretization of the stream function form
  of the {N}avier-{S}tokes equations.
\newblock {\em Comput. Math. Appl.}, 36(2):117--127, 1998.

\bibitem{fairag2003numerical}
F.~Fairag.
\newblock Numerical computations of viscous, incompressible flow problems using
  a two-level finite element method.
\newblock {\em SIAM J. Sci. Comp.}, 24(6):1919--1929, 2003.

\bibitem{fischer2022nekrs}
P.~Fischer, S.~Kerkemeier, M.~Min, Y.-H. Lan, M.~Phillips, T.~Rathnayake,
  E.~Merzari, A.~Tomboulides, A.~Karakus, N.~Chalmers, et~al.
\newblock {NekRS, a GPU-accelerated spectral element Navier--Stokes solver}.
\newblock {\em Parallel Computing}, page 102982, 2022.

\bibitem{foster2013two}
E.~L. Foster, T.~Iliescu, and D.~R. Wells.
\newblock A two-level finite element discretization of the streamfunction
  formulation of the stationary quasi-geostrophic equations of the ocean.
\newblock {\em Comput. Math. Appl.}, 66(7):1261--1271, 2013.

\bibitem{foster2016conforming}
E.~L. Foster, T.~Iliescu, and D.~R. Wells.
\newblock A conforming finite element discretization of the streamfunction form
  of the unsteady quasi-geostrophic equations.
\newblock {\em Int. J. Numer. Anal. Mod.}, 13(6), 2016.

\bibitem{galdi2011introduction}
G.~P. Galdi.
\newblock {\em {An introduction to the mathematical theory of the Navier-Stokes
  equations: Steady-state problems}}.
\newblock Springer Science \& Business Media, 2011.

\bibitem{gaonkar2015application}
A.~K. Gaonkar and S.~S. Kulkarni.
\newblock {Application of multilevel scheme and two level discretization for
  POD based model order reduction of nonlinear transient heat transfer
  problems}.
\newblock {\em Comput. Mech.}, 55(1):179--191, 2015.

\bibitem{GR86}
V.~Girault and P.-A. Raviart.
\newblock {\em Finite element methods for {N}avier-{S}tokes equations},
  volume~5 of {\em Springer Series in Computational Mathematics}.
\newblock Springer-Verlag, Berlin, 1986.
\newblock Theory and algorithms.

\bibitem{grosjean2022variations}
E.~Grosjean.
\newblock {\em Variations and further developments on the Non-Intrusive Reduced
  Basis two-grid method}.
\newblock PhD thesis, Sorbonne universit{\'e}, 2022.

\bibitem{grosjean2022doubly}
E.~Grosjean and Y.~Maday.
\newblock {A doubly reduced approximation for the solution to PDE's based on a
  domain truncation and a reduced basis method: Application to Navier-Stokes
  equations}.
\newblock 2022.

\bibitem{grosjean2022error}
E.~Grosjean and Y.~Maday.
\newblock {Error estimate of the non-intrusive reduced basis (NIRB) two-grid
  method with parabolic equations}.
\newblock 2022.

\bibitem{hesthaven2015certified}
J.~S. Hesthaven, G.~Rozza, and B.~Stamm.
\newblock {\em Certified Reduced Basis Methods for Parametrized Partial
  Differential Equations}.
\newblock Springer, 2015.

\bibitem{HLB96}
P.~Holmes, J.~L. Lumley, and G.~Berkooz.
\newblock {\em Turbulence, Coherent Structures, Dynamical Systems and
  Symmetry}.
\newblock Cambridge, 1996.

\bibitem{layton1993two}
W.~Layton.
\newblock A two-level discretization method for the {N}avier-{S}tokes
  equations.
\newblock {\em Comput. Math. Appl.}, 26(2):33--38, 1993.

\bibitem{layton1999two}
W.~Layton and X.~Ye.
\newblock {Two-level discretizations of the stream function form of the
  Navier-Stokes equations}.
\newblock {\em Numer. Funct. Anal. Optim.}, 20(9-10):909--916, 1999.

\bibitem{layton2008introduction}
W.~J. Layton.
\newblock {\em Introduction to the numerical analysis of incompressible viscous
  flows}, volume~6.
\newblock Society for Industrial and Applied Mathematics (SIAM), 2008.

\bibitem{liakos2003two}
A.~Liakos and H.~Lee.
\newblock Two-level finite element discretization of viscoelastic fluid flow.
\newblock {\em Comput. Methods Appl. Mech. Engrg.}, 192(44-46):4965--4979,
  2003.

\bibitem{quarteroni2015reduced}
A.~Quarteroni, A.~Manzoni, and F.~Negri.
\newblock {\em Reduced Basis Methods for Partial Differential Equations: An
  Introduction}, volume~92.
\newblock Springer, 2015.

\bibitem{sell2013dynamics}
G.~R. Sell and Y.~You.
\newblock {\em Dynamics of evolutionary equations}, volume 143.
\newblock Springer Science \& Business Media, 2013.

\bibitem{volkwein2013proper}
S.~Volkwein.
\newblock Proper orthogonal decomposition: Theory and reduced-order modelling.
\newblock {\em Lecture Notes, University of Konstanz}, 2013.
\newblock
  \url{http://www.math.uni-konstanz.de/numerik/personen/volkwein/teaching/POD-Book.pdf}.

\bibitem{wang2011two}
Z.~Wang, I.~Akhtar, J.~Borggaard, and T.~Iliescu.
\newblock Two-level discretizations of nonlinear closure models for proper
  orthogonal decomposition.
\newblock {\em J. Comput. Phys.}, 230:126--146, 2011.

\bibitem{Xu94}
J.~Xu.
\newblock A novel two-grid method for semilinear elliptic equations.
\newblock {\em SIAM J. Sci. Comput.}, 15(1):231--237, 1994.

\bibitem{Xu96}
J.~Xu.
\newblock Two-grid discretization techniques for linear and nonlinear {PDE}s.
\newblock {\em SIAM J. Numer. Anal.}, 33(5):1759--1777, 1996.

\end{thebibliography}

\end{document}